\documentclass{conm-p-l}
\usepackage{color,epsfig,graphics,subcaption}
\usepackage[utf8]{inputenc}
\usepackage{hyperref}

\newtheorem{theorem}{Theorem}[section]
\newtheorem{lemma}[theorem]{Lemma}

\newtheorem{conjecture*}[theorem]{Conjecture}

\theoremstyle{definition}

\newtheorem{proposition}[theorem]{Proposition}

\theoremstyle{remark}

\numberwithin{equation}{section}

\DeclareRobustCommand*{\vec}[1]{\overrightarrow{#1}}

\newcommand{\R}{{\mathbb{R}}}
\newcommand{\C}{{\mathbb{C}}}
\newcommand{\N}{{\mathbb{N}}}

\def\ha{\frac{1}{2}}

\def\pa{\partial}
\def\ra{\rightarrow}
\def\ga{\alpha}

\def\gd{\delta}

\def\gl{\lambda}
\def\go{\omega}

\def\gs{\sigma}

\def\OPD{pseudo-differential operator}

\begin{document}
\title{On gravito-inertial surface waves}

\author{Yves Colin de Verdi\`ere}
\address{Universit\'e Grenoble-Alpes, CNRS, Institut Fourier, 38000 Grenoble (France)}
\curraddr{}
\email{yves.colin-de-verdiere@univ-grenoble-alpes.fr}

\author{J\'er\'emie Vidal}
\address{CNRS, ENS de Lyon, Univ. Lyon 1, Laboratoire de G\'eologie de Lyon (France)}
\email{jeremie.vidal@ens-lyon.fr}

\date{Submitted February 2024}

\dedicatory{This paper is dedicated to the memory of Steve.}



\begin{abstract}
In geophysical environments, wave motions that are shaped by the action of gravity and global rotation bear the name of gravito-inertial waves. 
We present a geometrical  description of gravito-inertial surface waves, which are low-frequency waves existing in the presence of a solid boundary.
We consider an idealized fluid model for an incompressible fluid enclosed in a smooth compact three-dimensional domain, subject to a constant rotation vector. 
The fluid is also stratified in density under a constant Brunt-V\"ais\"al\"a frequency.
The spectral problem is formulated in terms of the pressure, which satisfies a Poincar\'e equation within the domain, and a Kelvin equation on the boundary.
The Poincar\'e equation is elliptic when the wave frequency is small enough, such that we can use the Dirichlet-to-Neumann operator to reduce the Kelvin equation to a pseudo-differential equation on the boundary. 
We find that the wave energy is concentrated on the boundary for large covectors, and can exhibit surface wave attractors for generic domains. 
In an ellipsoid, we show that these waves are square-integrable and reduce to spherical harmonics on the boundary. 
\end{abstract}

       \maketitle
\section{Introduction}
\subsection{The physics of the problem}
Global rotation and buoyancy are ubiquitous ingredients of many geophysical systems (e.g. the Earth's oceans, or the liquid cores of planets).
When the density stratification is stable (i.e. when a light fluid lies above a denser one), the Coriolis and buoyancy forces can sustain waves,
called {\it gravito-inertial waves}.
Since these waves are often believed to be key to understanding the dynamics of geophysical flows, they have received considerable interest. 
In the  linear approximation, these waves are governed by a mixed hyperbolic-elliptic equation \cite{friedlander1982jfm}, which is called the Poincar\'e equation below.
Indeed, the latter reduces to the so-called Poincar\'e equation (denoted by Cartan \cite{cartan1922petites} after the seminal work of Poincar\'e \cite{poincare1885equilibre}) for pure inertial waves without density stratification \cite{CdV2023spectrum}.
Let us consider the simplest case (known as the $f-$plane approximation in geophysical modeling), where the fluid is subjected to a constant global rotation $\vec{\Omega} = (f/2) \vec{e}_3$ ($f$ is called the Coriolis parameter in the geophysical litterature)
and stratified in density with a constant Brunt-Väisälä frequency $N$ under the constant gravity $\vec{g} = -g \vec{e}_3$.
For  fluids in the space $\R^3$, the  frequency $\omega \in \mathbb{R}\setminus 0$ of gravito-inertial waves is then given by the dispersion relation 
\begin{equation}
\label{eq:dispersionrelation}
    \omega ||\boldsymbol{\xi}||^2 = N^2 (\xi_1^2 + \xi_2^2) + f^2 \xi_3^2, \quad \boldsymbol{\xi} = (\xi_1, \xi_2, \xi_3)
\end{equation}
where $\boldsymbol{\xi} $ is the wave vector. 
Dispersion relation (\ref{eq:dispersionrelation}) shows that gravito-inertial waves only exist when $\min(N,f) \leq |\omega| \leq \max(N,f)$.  
Hence, in this case, there is a frequency gap $|\go |<\min(N,f)$  without bounded waves.
However, low-frequency gravito-inertial waves can exist in this frequency gap when the fluid domain is bounded \cite{friedlander1982jfm}.
Motivated by geophysical applications, we recently revisited this canonical problem and (surprisingly) found these low-frequency waves are polynomials when the boundary is an ellipsoid \cite{vidal2024igw}. 

The goal of this study is to present a geometrical description of these gravito-inertial surface waves, which are denoted below by Kelvin waves.
Indeed, they share similar properties with surface waves studied by Lord Kelvin in oceanography \cite{thomson18801}.
This is based on a reduction to the boundary of the wave equations, following our previous works on inertial \cite{CdV2023spectrum} and gravito-inertial waves \cite{vidal2024igw}.

\subsection{Introduction to the equations}

We consider an incompressible fluid in a smooth compact   domain $D\subset \R^3$, which is subjected to a constant global rotation.
The fluid is also supposed to be stably stratified in density with a constant Brunt-V\"ais\"al\"a frequency $N$.

%

Let  $\vec{u}$ be the velocity of the fluid with respect to a frame rotating with speed $\vec{\Omega }$,
$\rho $ be the density of the fluid and $\phi $ be the pressure.
Linear waves, called ``gravito-inertial waves ``, of frequency $\go $ satisfy
a system of equations inside the domain $D$ and boundary conditions. This system of equations is equivalent to
a scalar  equation  $P_\go \phi=0$ for the pressure $\Phi$. 
This equation is called the ``Poincar\'e equation''.
Another approach is to introduce an operator ${\mathcal P} $ acting on the pairs $(\vec{u},\rho )$  so that $P_\go $ is invertible in $L^2(D) $
if and only if $\go $ is not in the spectrum of ${\mathcal P}$.
The operator ${\mathcal P}$, which we call the ``Poincar\'e operator'', is a pseudo-differential operator of degree $0$.
In order to solve the previous spectral problem, one has to look at the ellipticity properties of ${ P}_\go$ or, equivalently, of ${\mathcal P}-\go $.
This splits into 2 parts: either $ P_\go $ is non-elliptic and $\go $ is in the spectrum of ${\mathcal P}$, or $P_\go $ is elliptic and one has to look at the ellipticity of the boundary conditions. 
These latter conditions allow us to reduce the spectral problem to a pseudo-differential equation $K_\go\phi =0 $ for the restriction of the pressure to the boundary. We will call this equation the ``Kelvin equation''.
Our goal is to describe in a more precise way these operators and in particular to compute the principal symbol of $K_\go $.

The manuscript is organized as follows.
We first present general properties of the Poincar\'e and Kelvin equations in \S\ref{sec:poincare} and \S\ref{sec:BVP}.
Then, a detailed derivation of the Kelvin equation is given in \S\ref{sec:comput}.
Next, a microlocal analysis of the Kelvin equation is presented in \S\ref{sec:special} and \S\ref{sec:attractors}.
Finally, the case where $\partial D$ is an ellipsoid is considered in \S\ref{sec:ell}. 
Notably, we show that the spectrum is pure point with polynomial eigenvectors, and that the pressure field reduces to a spherical harmonic on $\partial D$.

\section{The Poincar\'e operator}
\label{sec:poincare}
We consider a fluid in a compact domain $D\subset \mathbb{R}^3$ with a smooth\footnote{Throughout the paper ``smooth'' means $C^\infty$.} boundary $\partial D $, and we equip $\mathbb{R}^3$ with the canonical (orthogonal) basis vectors $(\vec{e}_1,\vec{e}_2,\vec{e}_3)$.
We assume that the fluid domain is submitted to a constant global rotation $\vec{\Omega}$, and that the fluid is stably stratified in density with a  Brunt-V\"ais\"al\"a frequency $N>0$ under the imposed gravity
$\vec{g}=-g \vec{e}_3$.  
In this paper, we always assume that $\vec{g}$ and $N$ are constant. 
We seek small-amplitude periodic motions, with  frequency $\go$.
We refer the interested reader to \cite{vidal2024igw} for further details about the physical model. 
In the linear theory, the equations in the rotating frame are
\begin{equation}
\label{equ:gen}
i\go \vec{u} +2 \vec{\Omega }\wedge \vec{u}-\rho \vec{g}=-\nabla \phi, \quad \mathrm{div} \vec{u} = 0, \quad i\go \rho =N^2 u_3 /g,
\end{equation}
where $\go $ is the frequency. The velocity vector field $\vec{u} $ in the domain $D$  is tangent to the boundary $\partial D$.
The pressure is denoted by $\phi$, which has no other boundary condition than those coming from the velocity boundary conditions, and $\rho$ is the density perturbation. 
We can rewrite these equations in a symmetric form as
\begin{equation}
\label{equ:gen'} 
i\go \vec{u} + 2 \vec{\Omega } \wedge \vec{u}+N \rho_1 \vec{e}_3 =-\nabla \phi, \quad \mathrm{div} \vec{u} = 0, \quad i\go \rho_1= N u_3,
\end{equation}
where $\rho_1:=\rho g/N$ is a rescaled density of the fluid.

In order to be more precise concerning the functional spaces, we introduce the so-called {\it Helmholtz decomposition} (see \cite{zbMATH05931547}, Chapter III).
If we  denote by $\vec{L}^2 (D)$ the Hilbert space of vector field in $D$ and by $H^1 (D) $ the space of scalar functions whose gradients are
in $L^2$, there is an orthogonal decomposition
$\vec{L}^2 (D)=\vec{L}^2 _0 (D) \oplus \nabla H^1 (D) $
where
$\vec{L}^2 _0 (D)$ is the $L^2$-closure of smooth vector fields compactly suppported in the interior of $D$ and divergence free.
The Leray projector $\Pi $ is the orthogonal projection on the first factor. The operator $\Pi $ is a \OPD~belonging to the classe introduced by Louis Boutet de Monvel \cite{BdM71,grubb2008distributions} and the principal symbol  $\pi(x,\boldsymbol{\xi})$ of $\Pi$  is the orthogonal projector on $\ker \boldsymbol{\xi}$.
If $\vec{u}$ is a smooth vector field in $D$, $\Pi \vec{u}$ is a smooth divergence free vector field tangent to the boundary and, conversely, any such vector field is in $\vec{L}^2 _0 (D)$.

Then, we introduce the operator
${\mathcal Q}_\go $ acting on $\vec{L}^2_0 (D) \oplus L^2(D) $ as
\begin{equation}\label{equ:op}
  {\mathcal Q}_\go (\vec{u},\rho_1)= \left(\Pi (i\go \vec{u} + 2 \vec{\Omega } \wedge \vec{u}+N \rho_1 \vec{e}_3),\quad i\go \rho_1- N u_3 
  \right),
\end{equation}
such that Equation (\ref{equ:gen'}) is rewritten as  ${\mathcal Q}_\go (\vec{u},\rho_1)=0$. 
Defining  $\Pi_0$ the operator acting on a pair $(\vec{u},\rho )$ by
\begin{equation*}
\Pi_0 (\vec{u},\rho_1) :=(\Pi \vec{u} , \rho_1),
\end{equation*}
we introduce the Poincar\'e operator $\mathcal{P}$ defined as follows: 
\begin{equation}\label{equ:A}
{\mathcal{P}} :=i\Pi_0 A \Pi_0, \quad A= \left( \begin{matrix}0 & -2\Omega _3 & 2\Omega _2 & 0 \\
2\Omega_3& 0 &- 2\Omega _1 &0 \\
-2\Omega _2 & 2\Omega _1 &0 & N \\
0 & 0 &-N & 0 \end{matrix} \right),
\end{equation}
where $iA$ is self-adjoint on $L^2 (D,\C^4)$. 
We get the following
\begin{lemma} For $\go \ne 0$, 
  the invertibility of the operator $Q_\go $ defined in Equation (\ref{equ:op})  on $\vec{L}^2_0 (D) \oplus L^2(D) $ is equivalent to the fact
  that $\go $ is not in the spectrum of $\mathcal{P}$.
\end{lemma}

The eigenvalues of $iA$ are $\pm \go_\pm $ with
\begin{equation}
\label{equ:omega} 
\omega_\pm = \sqrt{\ha \left( N^2+ f^2 \pm \sqrt{(N^2+f^2)^2 -16 N^2 \Omega _3^2}\right)}, \end{equation}
where we have introduced $f:= 2\|\vec{\Omega } \|$.
Then, the spectrum of $\mathcal{P}$ (i.e. the set of values of $\omega$ for which $\mathcal{P} - \omega I$ is not invertible, where $I$ is the identity operator) is given by the following theorem.

\begin{theorem}
\label{theo:general}
The Poincar\'e operator $\mathcal{P}$ is a bounded self-adjoint operator in $L^2 (D, \mathbb{C}^4)$, and its spectrum $\gs(\mathcal{P})$ is given by
\begin{equation*}
\gs(\mathcal{P})=  [-\go_+, \go_+].
\end{equation*}
\end{theorem}

The proof will be given  in Section \ref{sec:full}.

Our main interest in this paper will be the study of $\gs(\mathcal{P})\cap ]-\go_-,\go_-[\setminus 0$.
The low-frequency waves in this interval are referred to in \cite{friedlander1982jfm} as class-II solutions.
This part of the spectrum is controlled by an equation $K_\go \phi=0 $ satisfied by the pressure on the boundary, which is called here the Kelvin equation.
The solutions of the Kelvin equation for large covectors are localized on the boundary.
This shows that these solutions are associated with surface waves of the Poincar\'e operator.

\section{Equations for  the pressure} 
\label{sec:BVP}
We want to recast the equations for the velocity $\vec{u}$ and the density $\rho_1$ as a boundary-value problem for the scalar pressure $\phi$. 
Indeed, this will ease the microlocal analysis of the problem. 
The pressure satisfies an equation in $D$, called the Poincar\'e equation, and boundary conditions given by the Kelvin equation.
We explain below how to obtain these two equations in the general case. 

\subsection{Poincar\'e equation}
When $\go \ne \pm \go_\pm $, we can solve the equation
\begin{equation}
\vec{u} = M_\go (i\nabla \phi)
\label{eq:rPrelationship}
\end{equation}
with $M_\go =j^\star (\go -iA)^{-1}j$ where $j:\R^3 \ra \R^4$ is the injection $\vec{u}\ra (\vec{u},0)^t$.
Then, we will use  the divergenceless condition ${\rm div}(\vec{u})=0$ to obtain the Poincar\'e equation $P_\go \phi =0$ for $\phi$.
We write, using the classical formula for the inverse of a matrix, 
\[ \vec{u} = \frac{1}{{\rm det }(M_\go )}{\rm div}\left(^t {\rm adj}(M_\go ) \nabla \phi \right) \]
  and, removing the scalar factor in front of the previous expression,  the equation  
\begin{equation*}
P_\go \phi =   \begin{vmatrix} 
 \go -i A& \begin{pmatrix}\pa_1 \\ \pa_2 \\\pa_3 \\0 \end{pmatrix} \\ 
 \begin{pmatrix} \pa_1  & \pa_2  &\pa_3  &0 \end{pmatrix} & 0
\end{vmatrix} \phi.
\end{equation*}
Note that this simple writing of the Poincar\'e equation using a  determinant does not appear to have been written in this way in the literature.
We will use again that kind of writing later in order to get the Kelvin equation. 
Expanding the determinant, we get
\begin{equation}
\label{equ:poincare} 
P_\go \phi = -\go \left( (\go^2 -N^2)\Delta +N^2 \frac{\pa^2}{\pa x_3^2} -4
 \left(\sum_{i=1}^3 \Omega _i \frac{\pa}{\pa x_i}\right)^2\right) \phi. 
\end{equation}
The principal symbol of $P_\go$ is
\begin{equation}\label{equ:sigmago}
\gs_\go (\boldsymbol{\xi}) = \go \left( (N^2-\go^2)\|\boldsymbol{\xi} \|^2 -N^2 \xi_3^2 +4 \left( \boldsymbol{\xi} .\vec{\Omega}\right)^2\right),
\end{equation}
where $\boldsymbol{\xi} = (\xi_1,\xi_2, \xi_3)$ is the covector of norm $||\boldsymbol{\xi}|| = (\xi_1^2 + \xi_2^2 + \xi_3^2)^{1/2}$.
We also define the real-valued matrix $\Sigma_\omega$, associated with the quadratic form $\sigma_\omega$, as
\begin{equation}\label{equ:Sigmago}
\sigma_\omega (\boldsymbol{\xi}) := \langle \Sigma _\go \boldsymbol{\xi}|\boldsymbol{\xi} \rangle
\end{equation}
where $\langle .| . \rangle$ is the canonical duality product.
If $0<\go <\go_- $,  $\Sigma _\go $ is a real-valued, symmetric, and strictly positive matrix. 
The Poincar\'e equation  is elliptic if and only if
$0< |\go |< \go_-$ or $|\go | > \go_+ $.

\subsection{Kelvin equation}\label{ss:kelvin}
Using again the equation   $\vec{u} = M_\go (i\nabla \phi)$, we  write that $\vec{u}$ is tangent to $\pa D $.
This allows us to obtain the Kelvin equation. 
Since this equation is quite complicated in general, we will only compute it in some special cases (see \S\ref{sec:comput}). 
If $0<|\go| <\go_-$, the Poincar\'e equation is elliptic and the Kelvin equation can be written
as a pseudo-differential equation $K_\go\phi=0 $ on the boundary $\partial D$, which satisfies Theorem \ref{theo:kelvin}.  
\begin{theorem}
\label{theo:kelvin}
The Poincar\'e equation is elliptic when $0<|\go | <\go_-$, and the Kelvin equation reads then 
\begin{equation}
K_\go \phi = DtoN_\go \phi  +i \vec{W}_\go \phi,
\label{eq:kelvinth}
\end{equation}
where the Dirichlet to Neumann operator $DtoN_\go $ (see Appendix \ref{app:dtn}) is defined using the Euclidian metric $g_\go $ associated to the principal symbol $\gs_\go $ of $P_\go$, that is $g_\go =\Sigma _\go^{-1}$, and $\vec{W}_\go$ is a real-valued vector field tangent to the boundary.
Moreover, $\vec{W}_\go$ is divergenceless with respect to the area $v_\go $ associated to the restriction $g_{\go,\pa }$ of $g_\go $ to $\pa D $.
In particular, $K_\go $ is a self-adjoint \OPD ~ of degree $1$ on $L^2(\pa D, v_\go )$ whose principal symbol is 
\begin{equation}
k_\go (x,\boldsymbol{\xi})= \sqrt{g_{\go,\pa} ^\star(x,\boldsymbol{\xi})} - \langle \boldsymbol{\xi} | \vec{W}_\go(x)\rangle,
\end{equation}
where $(x, \boldsymbol{\xi}) \in T^\star(\partial D)$ and $g_{\omega,\partial}^\star$ is the dual of $g_{\omega,\partial}$.
The Kelvin equation $K_\go \phi =0$ is elliptic at $x\in \pa D$ if and only if $\| \vec{W}_\go (x)\| _{g_{\go,\pa}  }< 1$. 
\end{theorem}
Explicit computations of the vector field $\vec{W}_\go$ will be given in sections \ref{sec:comput} and  \ref{sec:special}. 

\begin{proof}
For any two smooth functions $\phi_1$ and $\phi_2$ on $D$,  let us  consider  the Hermitian bracket
\[ (( \phi_1|\phi_2 ))= \langle i \nabla \phi_1 | M_\go(i\nabla \phi_2) \rangle_{L^2} \]
where $\langle u_1 | u_2 \rangle_{L^2} = \int_D u_1 \cdot \overline{u}_2 \, |dx|$
where $\overline{u}$ is the complex conjugate of $u$.
Integrating by part, we get
\begin{equation}\label{equ:lgo}
(( \phi_1|\phi_2 ))=- \int_D \phi_1 \overline{{\rm div} \left(M_\go \nabla \phi_2\right)} \, |dx|   + \int _{\pa D } \phi_1 \overline{L_\go {\phi}_2} \, v_\go
\end{equation}
where $L_\go$ is a complex-valued vector field. 
Let us assume now that $P_\go \phi_2 = 0 $, and we get
\begin{equation*}
(( \phi_1|\phi_2 ))= \int _{\pa D } \phi_1 \overline{L_\go {\phi}_2 } \, v_\go
\end{equation*}
This vanishes if and only if $ M_\go (\nabla \phi_2)$ is orthogonal to all gradient fields and, hence, is tangent to the boundary. 
In other words, the equation $L_\go \phi =0$ is the boundary condition for $\phi $, which  is a starting point to get the Kelvin equation.
Assuming also that  $P_\go \phi_1 = 0 $, we see that $ L_\go $ is self-adjoint. We can always decompose the vector field $L_\go $
as $\Psi\vec{n}_\go + i\vec{W}_1 $ where $\Psi$ is a function,  $\vec{n }_\go $ is the unit outgoing normal to $\pa D$ for the metric $g_\go $, and $\vec{W}$ is a complex vector field tangent to
$\partial D$.
Then, we can rewrite $L_\go$ as 
$K_\go = \Psi DtoN_\go + i\vec{W}_1$. 
The self-adjointness of the operator $K_\omega$  implies that the principal symbol is real-valued.
$\Psi$ is thus real-valued, and $\vec{W}_1 $ is a real-valued vector field.
Taking the adjoint, we get
\begin{equation*}
K_\go=K_\go^\star = DtoN_\go \Psi + i\vec{W}_1 + i{\rm div}_\go(\vec{W}_1).
\end{equation*}
This implies, by separation of the real and imaginary parts, that ${\rm div}_\omega (\vec{W}_1)=0$ and $ \Psi DtoN_\go - DtoN_\go \Psi = 0$.
The last equation implies that $\Psi:\partial D \ra \R $ is invariant by the geodesic flow of
$g_{\go ,\pa }$ and, hence, constant. 
Removing this real constant, we get the final result.
\end{proof}

\subsection{What do the Poincar\'e and Kelvin equations say about $\mathcal{P}$?}\label{ss:poke}

Recall the well-known fact that the zeroes of the characteristic polynomial of a linear operator in a space of finite dimension are exactly the eigenvalues of the operator. 
This is no longer the case here.
We can think of the ``Poincar\'e + Kelvin'' equations as a kind of characteristic polynomial for the Poincar\'e operator.
A motivation for that is that the determinant of the principal symbol of ${\mathcal P}-\go $ is, up to some non-vanishing factor outside $\go =0$, the principal symbol $\gs_\go $ of $P_\go$.
More precisely, we have $ \gs_p \left( {\mathcal P} -\go \right) =-\go \gs_\go $.

In this section, we will focus on the case $0<|\go |< \go_- $. 
The case $\go_- < |\go | <\go_+ $ can be treated in a similar, even easier, way. 
Indeed, the ellipticity of the Poincar\'e equation and of $\mathcal{P}-\go $ are equivalent in the latter case.
Here, we have the following two results:
\begin{proposition} \label{prop:inj}
  The frequency $ \go $ is an eigenvalue of ${\mathcal P} $ with an eigenvector in $L^2$    if and only if
  there exists a non trivial  solution $\phi \in H^1$ of $P_\go \phi=0 $ and $K_\go \phi =0$.
\end{proposition}
This is clearly a consequence of the fact that one can calculate $\nabla \phi $ in terms of $(\vec{u},\rho_1)$ and vice-versa.

\begin{proposition}\label{prop:inv}If $0<|\go|< \go_-$ and  assume that $K_\go $ is non elliptic at some covector of $T^\star \pa D \setminus 0$,
  then $\go $ belongs to the spectrum of   ${\mathcal P}$.
\end{proposition}

We will use results on traces from distributions in $D$ to distributions in $\pa D$ which can be found in \cite{zbMATH03353865}, Chap. 8 and 9.
The idea is to use Weyl's criterion, that is to construct approximate eigenfunctions. 

We start with a sequence $ \eta_n \in H^\ha (\pa D)$\footnote{The spaces $H^s$ with $s\in\R$ are the classical Sobolev spaces.} so
that  $\| K_\go \eta_n \|_{H^{-\ha}}\leq \frac{1}{n}$ and $\| \eta_n \|_{H^{\ha}}=1 $. Such a sequence exists because $K_\go $ is non elliptic at least at some
covector in  $T^\star \pa D \setminus 0$.
We then extend $\eta_n $ as $\phi_n \in H^1(D)$ with
$P_\go \phi_n =0$ and $0<c\leq \| \phi_n \|_{H^1} \leq C $.
From the  sequence $(\phi_n)$, we  reconstruct $(\vec{u}_n,(\rho_1)_n )$ and will apply  the Weyl's criterion.
First, we get from the explicit formula giving $\vec{u}_n$ in terms of $\phi_n $  that 
$\|\vec{u}_n \| _{L^2} $ does not tends to zero. We need to show that
$\| (\mathcal{P}-\go )(\vec{u}_n,(\rho_1)_n )\|_{L^2}$ tends to zero.
We can reread equation (\ref{equ:lgo}) as
\begin{equation}\label{equ:lgo2}
\langle \nabla \psi|\vec{u}_n \rangle =- \int_D \psi P_\go \phi_n \, |dx|   + \int _{\pa D } \psi K_\go {\eta}_n \, v_\go
\end{equation}
with $\psi \in H^1(D)$. 
The second term vanishes by construction, while the last one is $O(1/n)\| \psi \|_{H^1 }$. This inequality  follows from the fact
that $\psi_{|\pa D }\in H^{\ha}(\pa D)$ with $\| \psi \|_{H^\ha (\pa D )}\leq C \|\psi \|_{H^1 (D) } $. 
Let us define $v_n:=\Pi \vec{u}_n$.
We have hence, from Pythagore's Theorem, 
$\vec{u}_n -\vec{v}_n=\vec{r}_n $ with $\| \vec{r }_n \|_{L^2}=O(1/n)$.
Hence, we have
\begin{equation*}
\| i\go \vec{v}_n +2\vec{\Omega } \wedge \vec{v}_n -N (\rho_1)_n +\nabla \phi_n \|_{L^2}=O(1/n), \quad \| i\go (\rho_1)_n -N (v_3)_n \| _{L^2} =O(1/n),
\end{equation*}
and $\vec{v}_n \in L_0^2 (D)$.
This leads to
\begin{equation*}
\| (\mathcal{P}-\go) (\vec{v_n},(\rho_1)_n)^t\| _{L^2} =O(1/n),
\end{equation*}
while $ \| \vec{v}_n \|_{L^2} $ stays bounded below by a $>0$ constant.

\section{Computing the Kelvin equation}
\label{sec:comput}
Let us fix a point $m\in \pa D $, and we assume that $m$ has coordinates $(0,0,0)$ and  the tangent plane to $\pa D$ at $m$
is given by $a_1x_1+a_2x_2 + a_3 x_3=0$ where $\boldsymbol{a}=(a_1,a_2,a_3)$ is a non-zero outward normal covector of $D$.
We will write directly, that is without using the Helmhotz decomposition, the condition that $\vec{u}$ is tangent on $\partial D$. 
It can be written as a vanishing condition for the determinant given by
\begin{equation}
B_\go(\nabla \phi, \boldsymbol{a})= \begin{vmatrix} 
 \go -i A& \begin{pmatrix}\pa_1\phi \\ \pa_2\phi   \\ \pa_3\phi  \\0 \end{pmatrix} \\ 
\begin{pmatrix} a_1  & a_2  &a_3  &0 \end{pmatrix} &0
\end{vmatrix}.
\label{eq:hermitianB}
\end{equation}
The Hermitian form $B_\omega$ can be written as $B_\go =\Sigma _\go +i \Xi_\go$, where $\Xi_\go  $ is real and antisymmetric. 
We can rewrite $ B_\go(\nabla \phi, \boldsymbol{a}) $ as a complex-valued vector field acting on $\phi$.
Indeed, this vector field is proportional to $L_\go $ from Equation (\ref{equ:lgo}). This  gives
\begin{equation*}
 B_\go(\nabla \phi, \boldsymbol{a})=\left(\vec{N} +i \vec{W}_1\right)\phi
\end{equation*}
with $\vec{N}= \Sigma_\go (\boldsymbol{a}) $. 
Using the fact that $g_\go =\Sigma_\go ^{-1}$, we get 
\begin{equation*}
g_\go (\vec{N}, \vec{e}) = \langle \boldsymbol{a}|\vec{e} \rangle.
\end{equation*}
This confirms that $\vec{N}$ is $g_\go$-orthogonal to $\pa D$. We see that  the vector $\vec{N}$ is outward because, if $\vec{n}$ is an outgoing vector,
$ g_\go (\vec{N}, \vec{n})= \langle \boldsymbol{a}|\vec{n} \rangle$ is positive.  
Moreover, we have
\begin{equation*}
    g_\go (\vec{N}, \vec{N})= \langle  \boldsymbol{a}|\vec{N} \rangle = \sigma _\go (\boldsymbol{a}).
\end{equation*}
Hence, we have $\vec{N}=\sqrt{\gs_\go(\boldsymbol{a})} \vec{n}_\go $. 
Finally, we get the Kelvin equation
\begin{equation*}
    \sqrt{\sigma _\go (\boldsymbol{a})} DtoN_\go \phi + i \vec{W}_1 \phi =0,
\end{equation*}
which is of the form (\ref{eq:kelvinth}) with $\vec{W}=\vec{W}_1/ \sqrt{\sigma _\go (\boldsymbol{a})}$.

The Kelvin equation is non-elliptic at points where the tangent plane is given by $\boldsymbol{a}$ if and only if 
\begin{equation*}
\| \vec{W}_1 \|_\go^2 > \sigma _\go (\boldsymbol{a}).
\end{equation*}
At the points $x\in \pa D$ such that $\vec{W}_1 (x)\ne 0 $ for $\go=0$, this holds when $\go $ is small enough. 
Indeed, there is a coefficient $1/\go $ in front of $g_\go $ and a coefficient $\go $ in front of $\gs(\go)$. 
Note also that $DtoN_\go$ and $\vec{W}_1$ are  even functions of $\go $ such that, if $K_\go \phi =0$, we have $K_{-\go}\bar{\phi}=0$. 
This allows us to obtain, as expected, real-valued eigenvectors.  


\section{Explicit computations}
\label{sec:special} 
\subsection{General case}
\begin{figure}
\centering
\includegraphics[width=0.7\textwidth]{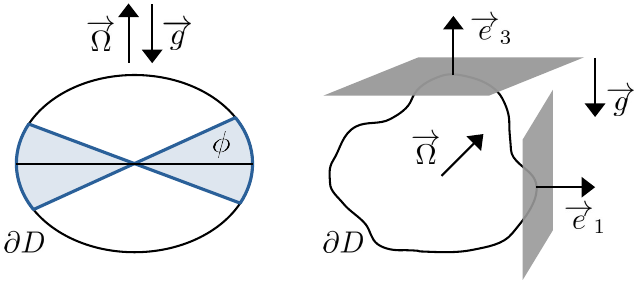}
\caption{Non-ellipticity of the Kelvin equation for different cases. \emph{Left}: Aligned case with $\vec{\Omega} \propto \vec{g}$. The Kelvin equation is non-elliptic near the equator when $\cos \phi \leq \omega/N$ (see \S\ref{ss:vertrot}). \emph{Right}: The Kelvin equation is elliptic when the tangent plane (gray) is parallel to gravity (see \S\ref{ss:strat-parallel}), whereas it is non-elliptic when the tangent plane is orthogonal to gravity (see \S\ref{ss:strat-orth}).}
\label{fig:geometry}
\end{figure}

As illustrated in figure \ref{fig:geometry}, we seek below whether the Kelvin equation is elliptic or not for different cases. 
To check the ellipticity of $K_\omega$, we need to compute the Hermitian form $B_\omega$ and the metric $g_\omega$.
Without loss of generality, we can assume that $\vec{\Omega}= (\Omega _1,0,\Omega _3 )^t $ with
\begin{equation} \label{equ:fsimple} f^2 = 4(\Omega _1^2 + \Omega _3^2). \end{equation}
The principal symbol of the Poincar\'e equation given in Equation (\ref{equ:sigmago}) is expanded  as
\begin{equation*}
\gs_\go (\boldsymbol{\xi})=\go \left [ (N^2+ 4\Omega_1^2-\go^2)\xi_1^2 +(N^2 -\go^2)\xi_2^2 + (4\Omega_3^2 -\go^2)\xi_3^2  +8\Omega_1 \Omega_3 \xi_1 \xi_3 \right ]
\end{equation*}
with $\boldsymbol{\xi}=(\xi_1,\xi_2,\xi_3)$, and the associated metric is
\begin{equation*}
g_\omega =\frac{1}{\gd \omega } \left ( (4\Omega_3^2 -\omega^2)dx_1^2 +(N^2+ 4\Omega_1^2-\omega^2)dx_3^2 -8\Omega_1 \Omega_3 dx_1 dx_3 \right )+\frac{dx_2^2}{\omega(N^2-\omega^2)}
\end{equation*}
with $\gd = (\go^2 - \go_-^2) (\go^2 - \go_+^2)$. 
Moreover, we have
\begin{equation*}
\vec{W}_1=2 \Omega _3 (N^2-\go^2) (a_1 \pa_2-a_2\pa_1)+2 \Omega_1 \go^2 (a_2\pa_3 -a_3\pa_2),
\end{equation*}
which is obtained from equation (\ref{eq:hermitianB}). 
We see that $\vec{W}_1$ is everywhere transverse to the meridians. 
This implies the wave propagation in the prograde direction (e.g. is eastward when $\Omega_3>0$). 
When $\omega$ is small, we also have
\begin{equation*}
    \vec{W}_1 = 2 \Omega _3 N^2 (a_1 \partial_2 - a_2 \partial_1) + \mathcal{O}(\omega^2)
\end{equation*}
and $\vec{W}= \vec{W_1}/\sqrt{\gs_\go (\boldsymbol{a})}$.
We can now evaluate these formulas for different cases.

\subsection{Vertical stratification and rotation}
\label{ss:vertrot}
We describe here the Kelvin equation when the rotation vector $\vec{\Omega }$ is parallel to gravity. We have  $f=2|\Omega_3|$
and, from Equation (\ref{equ:omega}),  $\go_-=\min (N,f),~\go_{+}=\max(N,f)$.
The corresponding Kelvin equation was derived in \cite{friedlander1982jfm}, and revisited in \cite{vidal2024igw} using microlocal arguments when $\partial D$ is an ellipsoid. We will use the quantity 
 $\gd= (\go^2 -N^2)(\go^2-f^2)$.  
The problem is then invariant by rotation around the $x_3$-axis.
Hence, we can assume that
$\boldsymbol{a}= (\cos\alpha, 0, \sin \alpha ) $ with $|\alpha |\leq \pi /2 $, where $\alpha$ can be called the latitude.
Then, we get
\begin{equation*}
\gs_\go (\boldsymbol{a}) =\go \left [ (N^2-\go^2) \cos^2 \alpha + (f^2 -\go^2) \sin^2 \alpha \right ]
\end{equation*}
and
\begin{equation*}
\vec{W}_1= f (N^2-\go^2)  \cos\alpha  \pa_2, \quad \|  \vec{W}_1 \|_\go ^2=\frac{f^2 (N^2-\go^2)^2 \cos ^2 \alpha}{\go (N^2-\go ^2)}.
\end{equation*}
We can check that the Kelvin equation is elliptic at large latitudes, that is when $|\cos \alpha|< \go / N $. 
In particular, the operator is never elliptic at the equator $\alpha =0$.

We can also compute the principal symbol  $k_\go $ of the Kelvin equation $K_\go\phi=0$.
This principal symbol depends only on the tangent plane to $\pa D$ at the point that we consider.
We can thus assume that
\[ D:\{ (x,y,z) | \cos  \alpha x + \sin \alpha z \leq 0 \}, \]
so that the co-vector $\boldsymbol{a}$ is indeed normal to the boundary and outgoing.
We parametrize the boundary by
\[ m(u,v)= (-v\sin \alpha , u, v\cos \alpha  )\]
and the dual momentum $(\xi_u,\xi_v)$.
The restriction of the metric $g_\go $ to $\pa D$ is given by
\[ g_{\go,\pa D}=\frac{1}{\go} \left( \frac{du^2}{N^2-\go^2} +\frac{((N^2-f^2)\cos^2 \ga +(f^2-\go^ 2))dv^2}{(N^2-\go^2)(f^2-\go^2)}\right). \]
Hence, the principal symbol of the DtoN operator is 
\[ \gs (DtoN)( \xi_u,\xi_v)=\sqrt{ \go \left((N^2-\go^2)\xi_u ^2 + \frac{(N^2-\go^2)(f^2-\go^2)\xi_v^2}{((N^2-f^2)\cos^2 \ga +(f^2-\go^2)}\right)}. \]
The principal symbol of the Kelvin equation is given, up to a product by a non-vanishing function, by 
\[ k_\go (\xi_u,\xi_v)=\sqrt{\gs_\go (\boldsymbol{a})(\xi_u,\xi_v)}\gs (DtoN)( \xi_u,\xi_v)-2\Omega _3 (N^2-\go^2)\xi_u  \cos \alpha  \]
or
\begin{equation}
  \label {equ:full} 
  \gs_\go (\boldsymbol{a})\gs (DtoN)( \xi_u,\xi_v)-f^2 (N^2-\go^2)\xi_u^2\cos ^2 \alpha, 
\end{equation}
with the condition that $\xi_u$ has the same sign as $\Omega _3$.
Expanding the expression (\ref{equ:full}), we get, after removing a non-vanishing factor, \[ \go^2 (\xi_u^2+\xi_v^2) -N^2 \cos^2 \alpha \xi_u^2 \] such that the principal symbol is given by 
\begin{equation*}
k_\go (x,\boldsymbol{\xi})=\go \mp  N \cos \alpha (x) \frac{\xi_u}{\sqrt{\xi_u^2+\xi_v^2 }},
\end{equation*}
where the sign is negative if $\Omega_3 >0$ and positive otherwise. 
In particular, we see that
\begin{itemize}
\item When $0<N\leq f $,  $T^\star \pa D \setminus 0 $ is foliated by the cones $k_\go^{-1}(\go )$ with $|\go |\leq \go_- =N $.
\item When $0<f<N $, we have 
\begin{equation*}
\cup_{|\go|\leq f }k_\go^{-1}(\go )=\left\{   \left| \cos \alpha \frac{ \xi_s}{\sqrt{\xi_s^2+\xi_u^2}}\right|\leq f/N \right\},
\end{equation*}
which is not the full cotangent space. 
This will have some implications on the Weyl asymptotics for the eigenvalues when $\partial D$ is an ellipsoid.
\end{itemize}

\subsection{Stratification parallel to the boundary}
\label{ss:strat-parallel}
We consider the case where the tangent plane is horizontal $\{ x_3 \leq 0 \}$, such that $\boldsymbol{a} = (0,0,1)$ and the stratification with $N>0$ is parallel to the tangent plane.
Then, we have
\begin{equation*}
\vec{W}_1= -2 \Omega_1 \go ^2  \pa_2, \quad \| \vec{W}_1 \|_\go^2 = \frac{ 4\Omega_1^2 \go^4 }{\go (N^2-\go^2)}, \quad \gs_\go (\boldsymbol{a})=\go (4\Omega _3^2 -\go^2).
\end{equation*}
The ellipticity condition is  given by
\begin{equation*}
\| \vec{W}_1 \|_\go^2< \go (4\Omega _3^2 -\go^2).
\end{equation*}
After some calculations, this gives
\[ (\omega^2-\omega_-^2)(\omega^2-\omega_+^2) >0 \]
when $0 < |\go |<\go_-$. 
Therefore, there are no surface waves in this case. 

\subsection{Stratification orthogonal to the boundary}
\label{ss:strat-orth}
We assume that $\boldsymbol{a}=(1,0,0)$ and $\Omega_3 \neq 0$. 
Then, we have 
\begin{equation*}
\vec{W}_1=  2 (N^2-\go^2)  \Omega_3 \pa_2, \quad \| \vec{W}_1 \|_\go^2= \frac{4 (N^2-\go^2) \Omega_3^2 }{\go }, \quad \gs_\go (\boldsymbol{a})= \go (N^2+ 4\Omega_1^2 -\go^2).
\end{equation*}
After some calculations, this gives the ellipticity condition
\[ (\omega^2-\omega_-^2)(\omega^2-\omega_+^2 )<0 .\]
We see that the Kelvin equation is non-elliptic when $0 < |\go |<\go_-$. 
Hence, there are always surface waves located near the equator.
This completes the fact that the spectrum is always $[-\go_+,\go_+] $, as stated in Theorem \ref{theo:general}.

\subsection{Proof of Theorem \ref{theo:general}} \label{sec:full}

We have $\| \Pi_0 \| =1$, while the eigenvalues of the Hermitian matrix $iA$ are $\pm \go_\pm $.
This proves that $\| \mathcal{P} \| \leq \go_+ $. 
The theorem follows from the  ellipticity conditions (see Appendix \ref{app:ell} and also Theorem 2.1. of \cite{colin2020spectral}).
A previous remark is that we can check that the determinant of the principal symbol of $\mathcal{P}-\go $ is equal to $\gs_\go $ up to some non-vanishing constant outside $\go=0$.

More precisely,
the proof of Theorem \ref{theo:general} splits into 3 parts:

\begin{enumerate}
  \item if $\go_-< |\go |<\go_+ $, the Poincar\'e operator himself is non-elliptic and we can use Weyl's criterion.
\item if $\Omega _3 \ne 0$ and  $0< |\go |<  \go_- $ , it follows from Section \ref{ss:strat-orth} that the operator
$K_\go $ is non elliptic at least at points where the stratification is orthogonal to the boundary. Such points
exist because $\pa D$ is smooth and compact: we can assume that  the stratification is horizontal, then they are the points where the tangent plan to the boundary is vertical.
\item if $\Omega _3=0$, $\go_-=0$ and there is nothing more to prove. 
\end{enumerate}

\section{Dynamics of waves}
\label{sec:attractors}
We can study a more general problem.
Let $(X,g)$ be a closed Riemannian surface, and $\vec{W}$ be a divergenceless vector field on $X$.
We consider a self-adjoint operator $K$ defined by
\begin{equation*}
K= E-\frac{1}{i}\vec{W},
\end{equation*}
where $E$ is an elliptic self-adjoint \OPD~of principal symbol $e$ of degree $1$. 
The principal symbol $k$ of $K$ is given by
\begin{equation*}
k (x,\boldsymbol{\xi})= e (x,\boldsymbol{\xi} )  -\langle \boldsymbol{\xi} | \vec{W}(x) \rangle.
\end{equation*}
If the characteristic cone $C_c: =k^{-1} (0)$ is non-degenerate, $K$ is called of principal type.
The projection on $X$ of $C_c$ is
$S=\pi_X(C_c) = \{   \| \vec{W} \| _g \geq 1 \} $.
For each $x \in \partial D$ such that $ \| \vec{W}(x) \| _g > 1 $, there are exactly two covectors directions in $C_c $. 
The base of the cone $C_c $ consists of
two copies of $S$ glued along the boundaries $\{   \| \vec{W} \| _g =1 \} $. 
If the Hamiltonian vector field is non-radial
on $C_c$, this base is the union of 2D-tori (because it supports a non-vanishing
vector field: the unscaled Hamiltonian field of the principal symbol).

If $k$ is allowed to smoothly vary as a function of a spectral parameter $\omega$ (e.g. $k_\go$ in the previous sections), the projection of the
classical (the group velocity) dynamics (the group velocity) onto $X$
is given by
\begin{equation*}
\left(\frac{\pa k_\go }{\pa \go }\right)^{-1} \left( \vec{G} + \vec{W} \right),
\end{equation*}
where $\vec{G}$ is the projection of the geodesic field for the corresponding values of $\xi $. The main direction is then given
by $\pm \vec{W}$ with $\pm =$sign $\frac{\pa k_\go }{\pa \go }$. 
This gives the motion to the east in our case, at least for small values of $\omega$. 

We can apply the study of attractors initiated in 
\cite{CdV2020attractors,colin2020spectral}. 
We assume that the Hamiltonian flow of $k$ satisfies the Morse-Smale assumptions of  \cite{CdV2020attractors}.
Then, the solutions of the Kelvin equation are Lagrangian distributions located over the attractors and the repulsors.
In the case of the Kelvin equation, we thus predict the existence of some attractors for the surface waves (which remain to be observed in physics).

\section{The case of ellipsoids}
\label{sec:ell}
\begin{figure}
\centering
\includegraphics[width=0.6\textwidth]{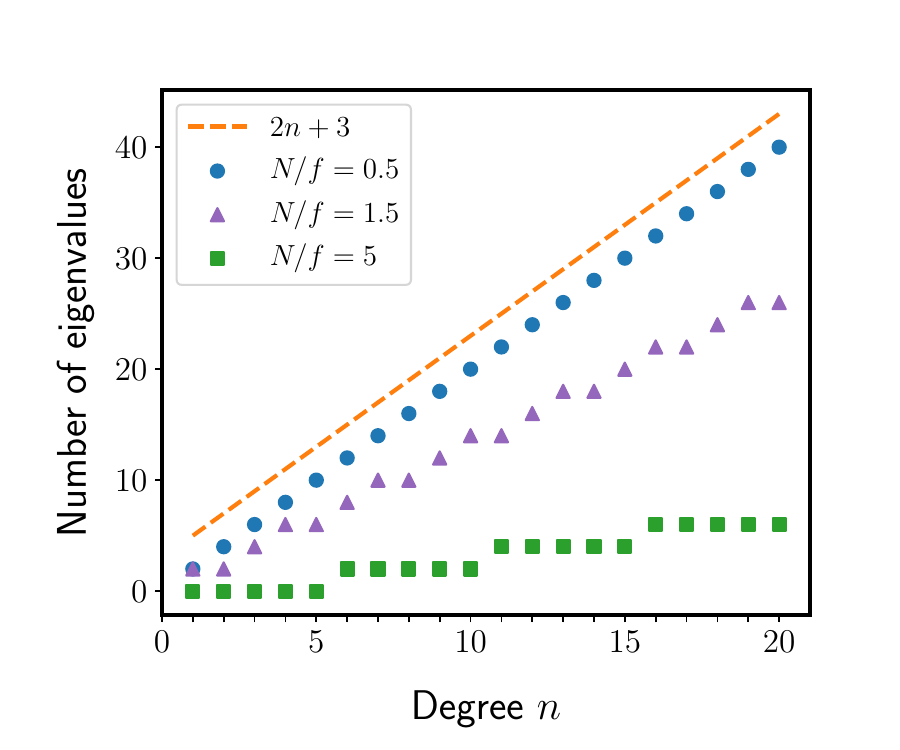}
\caption{Number of eigenvalues in the interval $0<|\omega|< \omega_-$ for polynomial eigenvectors of degree less than $n$. Aligned rotation and gravity (as considered in \S\ref{ss:vertrot}) when $\partial D$ is a sphere. 
For every degree $n$, the number of eigenvalues is bounded by the number  $2n+3$ of spherical harmonics  of degree  $n+1$. 
Numerical calculations following the method presented in \cite{vidal2024igw}.}
\label{fig:fig3}
\end{figure}

Now, we assume that $\partial D$ is an ellipsoid. We call then $D$ an ellipsoid body.
Ellipsoidal models have proven useful for geophysical applications (e.g. to model some peculiar vortices in the Earth's oceans \cite{vidal2024igw}).
In an ellipsoid body, the spectrum of the Poincar\'e operator is pure-point and dense in $[-\omega_+,\omega_+]$, with polynomial eigenvectors.
In particular, gravito-inertial surface modes were found in \cite{vidal2024igw} when $0 < |\omega| < \omega_-$. 
Numerical calculations (figure \ref{fig:fig3}) show that the number of eigenvalues associated with each polynomial vector space of degree $n$ is bounded by $2n+3$, which is the number of spherical harmonics of degree $n+1$.
Indeed, if the velocity $\vec{u}$ is of degree $n\geq 1$, then the pressure $\phi$ is of degree $n+1$ according to equation (\ref{eq:rPrelationship}).
This strongly suggests a link, when $\partial D$ is an ellipsoid, between the gravito-inertial surface modes and the theory of spherical harmonics.
This is the subject of this  section.
Moreover, the numerical results also confirm the microlocal predictions
detailed in \S\ref{ss:vertrot}, which show that the surface waves do not span the full cotangent space when $N\geq f$.
Another natural question is thus to estimate, when the polynomial degree $n$ tends to infinity, the asymptotic number of eigenvalues of the Poincar\'e operator in the interval $]0,\go_-[$.
A conjecture for this question is given in Section \ref{sec:perspectives}.

We consider the case of an ellipsoid body $E:=L(B)$, where $B$ is the unit Euclidean ball and $L$ is linear. 
The pull-back of the Poincar\'e operator to $B$ is then of the form ${\mathcal P}_E=\Pi_E C_E \Pi_E $ where $C_E$ is a $4\times 4$ matrix
and $\Pi_E$ is the Leray projector for the pull-back metric $L^\star (g_0)$. It turns out that ${\mathcal P}_E$ commutes with the Legendre operator
as used in  \cite{CdV2023spectrum}: 
recall that  the Legendre operator defined on functions with values  in $\C^d$ is defined by
\begin{equation*}
    \mathcal{L}^{\oplus d} := \left( -\Delta - L^\star L \right) \otimes {\rm Id}_{\C^d}
\end{equation*}
with $L=\sum_{i=1}^3 x_i \pa_i $, and  that $\mathcal{L}^{\oplus 4}$ commutes with the Poincar\'e operator. 

\begin{lemma}
$\mathcal{L}^{\oplus 3}$ and $\nabla$ satisfy a commutator formula given by
\begin{equation*}
    \mathcal{L}^{\oplus 3} \nabla - \nabla  \mathcal{L}^{\oplus 1} =-2\nabla (1+L).
\end{equation*} 
\end{lemma}
\begin{proof}
We have $[ \Delta , \nabla ]=0$, $[L,\nabla ]=-\nabla$ and $[L^\star ,\nabla ]=\nabla$. Then, the proof easily follows by using $L^\star =-L-3 $.
\end{proof}

\begin{lemma}
\label{prop:comm}
For any smooth function $\phi$ on $B$, we have
\begin{equation*}
\left(  \mathcal{L}^{\oplus 3} \nabla \phi \right)_{|S^2} =\nabla  \left(\Delta _{S^2} -2 \right) \phi_{|{S^2}}.
\end{equation*}
\end{lemma}
\begin{proof}
From the previous lemma, we have 
\begin{equation*}
    {\mathcal L}^{\oplus 3} \nabla \phi = \nabla \left( -\Delta  - L^\star L - 2(1+L) \right)\phi.
\end{equation*}
Using polar coordinates $(r,\theta )$ with $\theta \in S^2$, we get
\begin{equation*}
-\Delta  - L^\star L - 2(1+L) = -\frac{d^2}{dr^2}-\frac{2}{r}\frac{d}{dr}+\frac{1}{r^2}\Delta_{S^2} -L^\star L -2L -2.
\end{equation*}
Putting $r=1$ in the previous equation, and using $L^\star L= -r^2 \dfrac{d^2}{dr^2} -4L $,
we get $\left((-\Delta  - L^\star L - 2(1+L))\phi\right)  _{|S^2}= \left(\Delta _{S^2}-2 \right)\left(\phi _{|S^2} \right).$
\end{proof}

Hence, the Laplace-Beltrami operator on $S^2$ will play the role of $\mathcal{L}^{\oplus 3}$ for the waves in the bulk.
As a consequence, the gravito-inertial surface waves satisfy:
\begin{theorem}
\label{theo:harm}
The restriction to $\pa E$ of the pressure associated with every eigenvector of the Poincar\'e operator is a spherical harmonic. 
More precisely, if the ellipsoid body $E=L(B)$ where $B$ is the unit Euclidean ball and $L$ is linear, the pull-back of the pressures on the boundary of $E$ are $S^2$-spherical harmonics. 
\end{theorem}
\begin{proof}It follows from the equation $\nabla \phi =(\go -iA)(\vec{u},\rho_1)^t$ that, if $\vec{u}$ and $\rho_1 $  are of degree $n$ and have all components in $\ker \left({\mathcal L}-n(n+3)\right)$, ${\mathcal L}^{\oplus 3}(\nabla \phi )=n(n+3)\nabla \phi$.
Then, we use the commutation relation given in Lemma \ref{prop:comm} and get
\begin{equation*}
\nabla \left( \Delta _{S^2}\phi _{|S^2} - (n+1)(n+2) \phi _{|S^2}\right)=0.
\end{equation*}
The function $\phi _{|S^2} $ is a spherical harmonic of degree $n+1$ up to a constant (which can be assumed to vanish). 
Note that this holds for any eigenvector, without any restriction on the value of $\go $. 
\end{proof}

\section{Perspectives for future work}
\label{sec:perspectives}

Several questions have remained unanswered in the present study, and will be considered in future work. 
When $ D$ is an ellipsoid body, we are interested in a	more precise description of the low-frequency spectrum in the interval $]0,\omega_-[$.
It would be worth explicitly obtaining the number of eigenvalues with eigenvectors spherical harmonics of degree $n$ as well as  
 the asymptotic distribution of the eigenvalues in $]0,\omega_-[$  when the degree $n$ of the corresponding spherical
    harmonics tends to infinity. 
This would be complementary to the Weyl asymptotic formula in the interval $[\omega_-,\omega_+]$, which has been obtained in \cite{vidal2024igw}.
For $(x,\boldsymbol{\xi})\in T^\star (\partial D)\setminus 0$, let us denote by
$N(x,\boldsymbol{\xi})\in\N $ the number of values of $\omega$ such that $k_\omega (x,\boldsymbol{\xi})=0$ where $k_\go $ is the pull-back on $S^2$
of the principal symbol of
the Kelvin equation. We denote by ${U^\star S^2}$ the unit cotangent bundle of the 2-sphere with the canonical metric.  
Then, we conjecture the following asymptotics:
\begin{conjecture*}\label{conj:weyl-harm}
  Let $\mu_{j,n} $ be the eigenvalues of the Poincar\'e operator with polynomial eigenvectors of degree $n$.
Then, for $0<a,b<\go_-$,  we have 
\begin{equation*}
    \# \{ \mu_{j,n} \in [a,b]\} \sim \frac{2n+3}{8\pi^2} \int_{U^\star S^2\cap \{ (x,{\bf \xi})|k_\go (x,{\bf \xi})=0, ~ \omega \in [a,b]\}} N dL
\end{equation*}
when $n\ra +\infty $, where $L$ is the Liouville measure on $ U^\star S^2$.
\end{conjecture*}
{\it Idea of proof.--}
This two-dimensional Weyl formula could be proved using Bohr - Sommerfeld rules (following \cite{CdV80}).

Another interesting extension would be to consider an unstable stratification (i.e. when $N^2<0$). 
In this case, the Poincar\'e operator is no longer self-adjoint.
Then, it would be worth describing the essential spectrum (i.e. to determine the set of values of $\omega$ for which the boundary conditions are elliptic). 
This is a rather classical problem, which is for instance described in Chapter XX of \cite{hormander1985analysis}.
Finally, another interesting application would be the case where $\vec{g}$ is not a constant vector. 
For instance, this situation occurs inside planets or stars (where gravity is central). 
Then, for a given $\omega$, the Poincar\'e operator can be elliptic in some regions of $D$ and hyperbolic in others \cite{friedlander1982gafd}.
The calculus of the principal symbol of the Kelvin operator at $x\in \partial D$ depends only on the value of $\vec{g}$ at these points, so that the microlocal analysis presented in this work could be re-employed to investigate this problem. 

\appendix 

\section{The Dirichlet-to-Neumann operator} \label{app:dtn}
Let us consider a Laplace-Beltrami  operator $\Delta _g $ in $D$.
The Dirichlet-to-Neumann operator $DtoN $ acts on functions on $\pa D$ as follows.
If $f:\pa D \ra \C $ is smooth, one extends $f$ into a smooth function $F$ in $D$ satisfying  $\Delta _g F=0$ and $F_{| \pa D}= f$. 
Then, we define   $DtoN f = \pa F/\pa n_g $, where $n_g$ is the outgoing unit normal along the boundary.
It turns our that the  operator $DtoN$ is a self-ajoint \OPD~ of degree $1$ on $L^2 (\pa D, v_g )$,
where $v_g$ is the Riemannian volume on $\pa D$ with respect to the restriction $g_\pa$ of $g$ to $\pa D$.
In particular, $DtoN$ maps  the Sobolev space $H^\ha (\pa D)$  into the Sobolev space $H^{-\ha}(\pa D)$.
Note also that, if $f\in H^\ha (\pa D)$, the harmonic extension of $f$ in $D$ belongs to $H^1(D)$. 
The principal symbol of $DtoN_g $ is $\sqrt{g_\pa^\star}$, where $g_\pa^\star $ is the dual metric of $g_\pa$.
Further details are given in Chapter 7 of \cite{taylor2013partial}, and see also \cite{LU89,Gir22}. 

\section{Ellipticity and spectra}
\label{app:ell}
\begin{figure}
    \centering
    \includegraphics[width=0.35\textwidth]{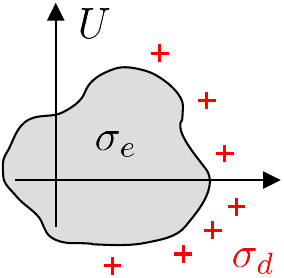}
    \caption{Discrete and essential spectra defined in Theorem \ref{theo:specelip}.}
    \label{fig:spectra}
\end{figure}

We denote by $\mathcal{P}$ the analytic family of pseudo-differential equations $P_\go$ (with $\go \in\C$) on a closed manifold $X$ with degree$(P_\go )\equiv 0$ (we can always assume the last condition by using a left product with an invertible elliptic operator). 
The operators $P_\go$ are continuous and linear from $L^2(X) $ into itself. 
The spectrum of this family is defined
as the set of $\go \in  \C $ so that $P_\go $ is not invertible. 
If $P_\go $ is invertible, the inverse is also continuous.
We also remind the reader that a \OPD~ $Q$ is elliptic at $(x,\boldsymbol{\xi})\in T^\star X \setminus 0$ if and only if the principal symbol $p(x,\boldsymbol{\xi}) $ is invertible. 
$Q$ is elliptic if this holds for all $(x,\boldsymbol{\xi}) \in T^\star X \setminus 0$. 
We will need the following theorem defining the different spectra (as illustrated in figure \ref{fig:spectra}).

\begin{theorem} 
Let $\mathcal{P}$ be as defined above, and $U\subset \C$ be the open set of the values of $\go$ where $P_\go$ is elliptic.
Then, there exists a discrete subset $\gs_d (\mathcal{P})$ of $U$ such that $P_\go$ is not invertible. 
If $\go \in\gs_d (\mathcal{P})$, then $\dim (\ker (P_\go))$ is finite.
More precisely, for all $\go \in U$,  $P_\go$ is Fredholm.
If  $\go\in \gs_e (\mathcal{P})=\C\setminus U$, then $P_\go$ is not invertible. 
\label{theo:specelip}
\end{theorem}

If $P_\go $ is elliptic, we can define a right parametrix as
\begin{equation*}
P_\go R_\go ={\rm Id} + A_\go
\end{equation*} 
with $A_\go $ smoothing and, similarly, a left parametrix $L_\go $.
Hence, $P_\go $ is Fredholm and we can apply the Fredholm analytic theorem. 
Conversely, if $P_\go $ is non-elliptic, it is non-invertible.
We can use test functions of the form $a(x)e^{i\gl x\boldsymbol{\xi}_0 }$ with $p_\go(a(x_0),\boldsymbol{\xi}_0)=0$, where $p_\go $ is the principal symbol of $P_\go$, in order to  show that the operator does not admit a continuous
inverse (see \cite{colin2020spectral}, section 2, for more details). 
This result also applies to the boundary calculus of \OPD's \cite{BdM71,grubb2008distributions}. 

It follows from this that the essential spectrum of the Poincar\'e operator is the same, in the interval $0<\omega < \omega_- $, as the set of values of $\omega$ for which the Kelvin equation is invertible. 
Both properties are equivalent to say that the boundary symbol is non-elliptic. 
This depends on the  symbolic calculus of boundary \OPD s, which is a rather technical issue. 
The identification of discrete spectra is much simpler.
Indeed, we can transfer any eigenmodes of the Poincar\'e operator to a $H^1$ solution of the Kelvin equation (and conversely).
More details will be given in future work, where we consider also the case of an unstable stratification (i.e. $N^2 < 0$) for which a greater part of that calculation is needed.

\bibliographystyle{alpha}
\bibliography{K-waves}

\end{document}